\documentclass[english,dvipsnames,11pt,oneside]{smfart}
\usepackage[french,main=english]{babel}
\usepackage[utf8]{inputenc}
\usepackage[T1]{fontenc}
\usepackage[a4paper]{geometry}

\usepackage{latexsym, amsfonts, amssymb, graphicx, smfthm,bbold}
\usepackage{mathtools}
\usepackage[hypertexnames=false]{hyperref}
\usepackage{bm}

\usepackage[all,dvips]{xy}
\usepackage{graphicx}
\usepackage{url}
\usepackage{enumerate}
\usepackage{tikz}
\usepackage[normalem]{ulem}

\theoremstyle{plain}

\newtheorem{theorem}{Theorem}\newtheorem{lemma}{Lemma}[section]
\newtheorem{proposition}[lemma]{Proposition}
\newtheorem{corollary}[lemma]{Corollary}

\theoremstyle{definition}

\theoremstyle{remark}
\newtheorem{remark}[lemma]{Remark}

\def\K{K}
\def\Z{{\mathbb Z}}

\def\L{{\mathcal L}}
\def\O{\mathcal{O}}

\newcommand{\F}{\mathbb{F}}

\newcommand{\map}[4]{\left\{
    \begin{array}{ccc}
      #1 & \longrightarrow & #2\\#3 & \longmapsto & #4
    \end{array}
  \right.}

\newcommand{\Pg}{P}
\newcommand{\Pinf}{P_{\infty}}
\newcommand{\Qg}{Q}
\newcommand{\Qinf}{Q_{\infty}}
\newcommand{\Uinf}{U_{\infty}}
\newcommand{\OKx}{\O_{\K(x)}}
\newcommand{\OL}{\O_{L}}
\newcommand{\OF}{\O_{F}}
\newcommand{\barS}{\overline{S}}
\newcommand{\barK}{\overline{\K}}
\newcommand{\barSz}{\overline{S}_0}
\newcommand{\barOL}{\overline{\mathcal{O}}_L}
\newcommand{\barOLS}{\overline{\OL S}}

\newcommand{\eqdef}{\stackrel{\text{def}}{=}}

\newcommand{\Fgal}{F'}
\newcommand{\Gal}{G}
\newcommand{\Dgal}{D'}
\newcommand{\Qgalinf}{Q_{\infty}'}
\newcommand{\dimK}{\dim_K\!}

\renewcommand{\geq}{\geqslant}
\renewcommand{\leq}{\leqslant}

\title{Freiman's $3k-4$ Theorem for Function fields}
\author{Alain Couvreur}
\address{Inria}
\address{Laboratoire LIX, \'Ecole polytechnique, Institut Polytechnique de Paris, 1~rue Honor\'e d'Estienne d'Orves, 91120 Palaiseau, France}
\email{alain.couvreur@inria.fr}

\author{Gilles Z\'emor}
\address{Institut de Mathématiques de Bordeaux, UMR 5251 and Institut
Universitaire de France}
\address{Université de Bordeaux, 351 cours de la libération, 33405 Talence, France}
\email{zemor@math.u-bordeaux.fr}

\begin{document}
\begin{abstract}
Freiman's $3k-4$ Theorem states that if a subset $A$ of $k$ integers has a Minkowski sum $A+A$ of size at most $3k-4$, then it must be contained in a short arithmetic progression. We prove a function field analogue that is also a generalisation: it states that if $K$ is a perfect field and if $S\supset K$ is a vector space of dimension $k$ inside an extension $F/K$ in which~$K$ is algebraically closed, and if the $K$-vector space generated by all products of pairs of elements of $S$ has dimension at most $3k-4$, then $K(S)$ is a function field of small genus, and $S$ is of small codimension inside a Riemann-Roch space of $K(S)$.
\end{abstract}

\maketitle

\section{Introduction}

We are interested in analogues of addition theorems that occur
in field extensions $F/\K$ of a base field $\K$. If $S$ and $T$ are
finite-dimensional $\K$-vector subspaces of $F$, we denote by $ST$ the
$\K$-linear span of the set of all products $st$, $s\in S$, $t\in T$.
The general purpose of this area of research is to characterise
subspaces $S$ and $T$ whose product $ST$ has unusually small
dimension: it is naturally inspired by one of the goals of additive
combinatorics which is to characterise subsets $A,B$
of a group that have sumsets $A+B$ of small cardinality, where $A+B$
denotes the set of elements $a+b$, $a\in A$, $b\in B$.

This line of research was arguably triggered by Hou, Leung and Xiang
when they proved \cite{hlx02} an extension field analogue of Kneser's
classical addition Theorem for Abelian groups~\cite{k56}. The
extension field version essentially states that, given a field extension $F/\K$, if $S$ and $T$ are
$\K$-linear spaces of finite dimension contained in $F$, then whenever $\dimK
ST<\dimK S +\dimK T -1$, it must be the case that the product space
$ST$ is stabilised by a non-trivial (different from $\K$) subfield of
$F$. This theorem was initially proved for separable extensions in
\cite{hlx02}, and was shown to be true for all extensions in
\cite{bsz18}.  Versions of Kneser's Theorem for products of vector
spaces were also shown to hold in other algebras than field extensions
\cite{bl15,mz15}.  Many applications of the theory of products of
spaces in the algebra $\F_q^n$ with componentwise multiplication are
discussed in \cite{r15}~: the latter theory is concerned in particular
with the existence and the construction of products of linear codes
of small dimension, and since many of these codes are constructed by
evaluating functions drawn from a function field over $\F_q$, this
makes the study of the dimension of products of spaces in function
fields relevant to this topic as well.

If $F/\K$ does not contain a finite extension of $\K$, and if
$S,T$ are $\K$-subspaces of $F$ such that $ST\neq F$, then Kneser's
Theorem for extension fields implies that $\dimK ST\geq \dimK S +\dimK
T -1$. A study of the limiting case when $\dimK ST=\dimK S + \dimK T
-1$ was undertaken in \cite{bsz17}. A particularly simple case occurs
when $\K$ is algebraically closed. One can then easily uncover the
structure of the spaces $S$ and $T$
choosing a discrete valuation $v$ on $F$, considering the sets
of integers $v(S)$ and $v(T)$ and arguing that their Minkowski sum $v(S)+ v(T)$
must satisfy $|v(S) + v(T)| = |v(S)| + |v(T)| - 1$,
from
which addition in $\Z$ tells us that
$v(S)$ and $v(T)$ must be arithmetic
progressions with the same difference: It can then be inferred that the spaces $S$ and $T$
must have bases in geometric progression.  
Switching to the symmetric case $S=T$ for the sake of
simplicity, it becomes then tempting to ask for the structure of
spaces $S$ such that
\begin{equation}
\label{eq:gamma} \dimK S^2 = 2\dimK S -1 +\gamma
\end{equation} for increasing values of $\gamma$. It was suggested in
\cite{bcz18} that one should draw inspiration from what is often known
as \emph{Freiman's ``$3k-4$'' Theorem} \cite{f73}, (called \emph{Freiman's $3k-3$
Theorem} in \cite[Thm.~5.11]{tv06}). Recall that it states that, if $A$
is a set of integers,
  \begin{equation}\label{eq:A+A} |A+A|=2|A|-1+\gamma
  \end{equation}
  implies, when $\gamma\leq |A|-3$, that $A$ is included in an
  arithmetic progression of length $|A|+\gamma$. Natural analogues of
  such sets $A$ are spaces $S$ that are of codimension at most
  $\gamma$ inside a space with a basis in geometric
  progression. However, the point was made in \cite{bcz18} that other
  examples of spaces satisfying \eqref{eq:gamma} are Riemann-Roch
  spaces $\L(D)$ of an algebraic curve of genus~$\gamma$.  Note that a
  space with a basis in geometric progression is a particular instance of a
  Riemann-Roch space in a genus $0$ function field.  In the present article
  we show that Riemann-Roch spaces play the role of arithmetic progressions in
  the extension field setting and that the following analogue of
  Freiman's $3k-4$ Theorem holds:
  
  \begin{theorem}\label{thm:main} 
    Let $\K$ be a perfect field and $F$
    be an extension field of $\K$ in which $\K$ is algebraically
    closed (i.e. any element of $F$ that is algebraic over
    $\K$ is in $\K$). Let $S$ be a $K$-subspace of finite dimension in
    $F$ such that $\K\subset S$ and $F = K(S)$. If
    \[
    \dimK S^2 = 2\dimK S -1 +\gamma \quad
    \text{with}
    \quad
    0\leq\gamma\leq \dimK S -3,
    \]
    then $F$ has transcendence degree $1$ over $K$ and its genus $g$
    satisfies $g\leq\gamma$. Furthermore, there exists a Riemann-Roch
    space $\L(D)$ of $F$ that contains $S$ and such that
    $\dimK \L(D) - \dim_\K S \leq \gamma -g$.
\end{theorem}

Theorem~\ref{thm:main} was first conjectured in \cite{bcz18} where the
statement was proved for $\gamma \leq 1$. Moreover, it was also proved that, under the conditions of the theorem, $F$ must always be of transcendence degree $1$ over $K$. Later on, the conjecture was proved
for $\gamma = 2$ under the additional assumption that $F = \K (x)$ \cite{w24}.

The integer $\gamma$ in \eqref{eq:gamma}
was called the \emph{combinatorial genus} of the space $S$ because
of the way it relates to the geometric genus $g$ of the associated
function field.  We note that Theorem~\ref{thm:main} is not only a
linear analogue of Freiman's $3k-4$ Theorem, but can also be seen a
generalisation since it easily recaptures the original additive
statement. Indeed, if $A$ is a set of integers satisfying
\eqref{eq:A+A} with $\gamma\leq |A|-3$, then, supposing without loss
of generality that $A\subset [0,n=\max A]$ and $\mathrm{gcd}(A)=1$, we
may form the vector space $S$ generated by $x^a$, $a\in A$, inside the
rational function field extension $\K(x)$. We have that $S^2$ is
generated by $x^a$, $a\in A+A$, and Theorem~\ref{thm:main} therefore
applies. The relevant Riemann-Roch space $\L(D)$ in the conclusion of
the Theorem can only be $\L(nP_\infty)$, where $P_\infty$ is the place
at infinity and $n = \max A$, and that $S$ is of codimension at most
$\gamma$ inside $\L(nP_\infty)$ means exactly that $A$ equals the arithmetic
progression $[0,n]$ with at most $\gamma$ missing elements. 

The way we prove Theorem~\ref{thm:main} borrows from Lev and
Smeliansky's elegant proof of Freiman's additive $3k-4$ Theorem
\cite{ls95} (see \cite[Th. 5.11]{tv06} for a concise description). The
idea was to project the set of integers under study onto a finite
quotient $\Z/n\Z$ and apply Kneser's Theorem in this finite group. Our
starting point will be to view $F = K(S)$ as a \emph{finite} extension
of a transcendental extension $K(x)$ of $K$, for a suitably chosen element
$x\in S$. The goal will then be to apply the extension field version
of Kneser's Theorem to $K(S)/K(x)$ and extract non-trivial information
from it: there will be significant differences with the additive situation
however. In the next section we give an overview of the paper and the
structure of the proof of Theorem~\ref{thm:main}, which will draw upon
a number of algebraic and combinatorial techniques.

 \section{Overview}

\Subsection{The additive case.}
To describe the main steps of the proof of Theorem~\ref{thm:main}, it
is useful to have in mind the template for Lev and Smeliansky's proof
of the original $3k-4$ Theorem. One starts with a set of integers $A$
satisfying~\eqref{eq:A+A} 
\[|A+A|=2|A|-1+\gamma\]
with $\min A = 0$, $\max A=n$ and
$\mathrm{gcd}(A)=1$: the goal is to prove that $|A|\geq
n+1-\gamma$. To this end one reduces the set $A$ to the set
$\Tilde{A}$ of integers modulo $n$, so that we have
$|\Tilde{A}|=|A|-1$. Since the sets $A$ and $A+n$ have only $n$ as a
common element, and since both reduce to $\Tilde{A}$ modulo $n$, we
have that
\begin{equation}\label{eq:card_tilde_A}
  |\Tilde{A}+\Tilde{A}|\leq |\Tilde{A}|+\gamma \leq 2|\Tilde{A}|-2.
\end{equation}
Now Kneser's Theorem tells us that there exists a subgroup $H$ of $\Z/n\Z$
such that $\Tilde{A}+\Tilde{A}=\Tilde{A}+\Tilde{A}+H$ and $|\Tilde{A}+\Tilde{A}|\geq 2|\Tilde{A}|-|H|$
(with \eqref{eq:card_tilde_A} therefore implying that $|H|>1$). It could happen that
$\Tilde{A} + \Tilde{A} = \Z/n\Z$, in which case we are done: indeed, the left-hand
inequality of (\ref{eq:card_tilde_A}) entails precisely that
$|\Tilde{A}| \geq n - \gamma$ and hence $|A| \geq n+1 -
\gamma$. 
The rest of the proof consists of showing that $H\neq \Z/n\Z$ cannot happen.
This is achieved by partitioning $A$ into sets that represent cosets modulo $H$:
by studying the behaviour of sums of these sets, one finds elements of $A+A$ that are equal modulo $H$ and counting them contradicts the hypothesis $\gamma\leq |A|-3$ in \eqref{eq:A+A}. The argument is detailed concisely in \cite{tv06}.

\Subsection{The function field setting.}
It is shown in \cite{bcz18} 
that under the hypothesis of the Theorem,
 the field $F$ must have transcendence degree $1$ so that it is a function field in one variable. 
 It is then natural enough to consider the smallest Riemann-Roch space $\L(D)$ containing $S$ as the analogue of the interval $[0,n]$ in the additive case. The divisor $D$ is therefore the least sum of all the places $P$, counted with suitable multiplicities, for which there exist functions in $S$ with poles at $P$. We then identify a function $x$ in $S$ whose pole divisor is exactly $D$. Next, we augment the space $S$ by considering its linear span $\Tilde{S}=K(x)S$ over $K(x)$. The choice of $x$ will ensure that $S\cap xS=Kx$ so that $\dim_K(S+xS)=2\dim S -1$. This guarantees that
\[
\dim_{K(x)} \Tilde{S}^2 \leq \dim_{K(x)}\Tilde{S} +\gamma \leq 2\dim_{K(x)}\Tilde{S} -2
\]
similarly to \eqref{eq:card_tilde_A}. Viewing now $F$ as an algebraic extension of $K(x)$ (rather than $K$), 
we can apply Kneser's Theorem for field extensions~\cite{bsz18} which says that there exists a field $L$,
$K(x)\subset L\subset F$, such that $\Tilde{S}^2=L\Tilde{S}^2$ and $\dim_{K(x)}\Tilde{S}^2\geq 2\dim_{K(x)}\Tilde{S}-\dim_{K(x)}L$. Since $2\dim_{K(x)}\Tilde{S}-\dim_{K(x)}L\leq 2\dim_{K(x)}\Tilde{S} -2$ we have that the field $L$ must be strictly larger than $K(x)$.
It will not be difficult to obtain that if $L=F$, then, somewhat similarly to the additive case, the conclusion of Theorem~\ref{thm:main} holds. This is a consequence of Clifford's Theorem and is detailed in Section~\ref{sec:L=F}.

It will remain therefore to show that the situation $L\neq F$ cannot occur. At this point, similarities with the additive case break down rather seriously and we are confronted with new challenges. 
Following the additive template would be easier if we could prove the existence of an $L$-adapted basis of $S$, i.e. the existence of elements $s_i\in S$ such that $LS$ decomposes as the direct sum of the $Ls_i$, and $S$ decomposes as the direct sum of $S_i=S\cap Ls_i$. We could then study the dimensions of the products $S_iS_j$ without fearing the appearance of unexpected linear relations between these subspaces.
We do not know how to argue for the existence of such a decomposition of $S$ however, but we did manage to obtain a partial result in this direction. We showed that the dimension over $K$ of $L\cap S$ is a non-trivial quantity: this is the object of Section~\ref{sec:ScapL}. The way this is achieved is by studying the evaluations of functions of $S$ on suitably chosen places, and showing that some of these functions must have block-wise constant evaluations, exactly as elements of $L$ do. Galois Theory arguments then enable us to show that these evaluations must come from functions of $S$ that also fall into $L$.

Section~\ref{sec:comb_proof} switches from the algebraic techniques developed in Section~\ref{sec:ScapL} to a much more combinatorial approach. It uses the result of Section~\ref{sec:ScapL} on $S\cap L$ to decompose $S$ into a direct sum of spaces that is much cruder than what would be provided by a hypothetical $L$-adapted basis, but which turns out to be sufficient to give an evaluation of $\dim_KS^2$ that only just contradicts the hypothesis \eqref{eq:gamma} of Theorem~\ref{thm:main}.

\Subsection{A geometric comment.}
Using $K(x)S$ as we do
may not seem like the most natural approach. 
Since the additive case resorted to a quotient of the ring of integers, a tempting analogue, 
that may seem closer in spirit to the additive proof, would be to consider a subring of $F$ containing $S$ and reduce it modulo some ideal.
In Section~\ref{sec:ScapL} this is exactly what we do: we choose a place $O$ of $K(x)$ which splits totally into a sum $Q$ of places in $F$ and we study $S$ modulo~$Q$. However, we have not seen how to derive a proof of Theorem~\ref{thm:main} using only this approach, even though there exist Kneser-type theorems for such quotients \cite{bl15,mz15}.

Nevertheless, we believe that using $\K(x)S$ definitely makes sense from a geometric point of view. 
Indeed, the extension $F/\K(x)$ corresponds to some curve covering
$\mathcal Y \rightarrow \mathcal X$, where $\mathcal X, \mathcal Y$
have respective function fields $\K (x)$ and $F$. The space $\K(x)S$ can be
interpreted as the restriction of $S$ to the generic fibre of the
covering $\mathcal Y \rightarrow \mathcal X$, while reducing
$S$ modulo $Q$ corresponds to a restriction to some special fibre.
In the end, proving the result required us to restrict $S$ both to the generic fibre and to some special fibre.

\Subsection{Organisation of the article}
In Section~\ref{sec:L=F} we provide the setup for the proof of Theorem~\ref{thm:main}, introduce the subextension fields of $K$, $K(x)\subset L\subset F$, and prove that the conclusion of the Theorem holds if $L=F$. In Section~\ref{sec:ScapL} we study the dimension over $K$ of $S\cap L$. In Section~\ref{sec:comb_proof} we conclude the proof by showing that $L\neq F$ cannot happen.
Both Sections~\ref{sec:ScapL} and \ref{sec:comb_proof} rely on the results of Section~\ref{sec:L=F}:
Section~\ref{sec:comb_proof} relies upon the result of Section~\ref{sec:ScapL} on the dimension of $S\cap L$, but it is possible to put off its proof and read Section~\ref{sec:comb_proof} before Section~\ref{sec:ScapL}, if one so wishes. 
The proof of Theorem~\ref{thm:main} developed through Sections~\ref{sec:L=F} to \ref{sec:comb_proof} uses the hypothesis that the ground field $K$ is algebraically closed: in Section~\ref{sec:perfect_field} we show how to extend the result to perfect fields, yielding Theorem~\ref{thm:main} in full generality.

\section{Setting up the proof of Theorem~\ref{thm:main}: the subfields $K(x)$ and $L$}
\label{sec:L=F}

\Subsection{Background and Notation}
\Subsubsection{The arithmetic of function fields}
Number theoretic background is fairly basic and most of it can be found in Chapters~1 and 3 of \cite{S09}, from which we borrow notation.

Given a field $K$,
the function fields over $K$ that we will encounter (see Section~\ref{subsec:transc_deg_1} below) will be exclusively function fields in one variable, i.e. finite extensions of the field $K(x)$ of rational functions.

In this setting, a \emph{valuation ring} of the function field $F$ is a proper subring 
$K \subsetneq \mathcal O \subsetneq F$ that is maximal for inclusion. Such rings are
discrete valuation rings, i.e. they are local rings whose unique maximal
ideal is principal. Following \cite{S09} and other authors we call \emph{places} of $F$ the maximal ideals of the discrete valuation rings of $F$. The induced discrete valuation of a place $P$ is denoted
$v_P : \K^{\times} \rightarrow \mathbb{Z}$.
In most of the article, with the notable exception of Section~\ref{sec:perfect_field}, the field 
$\K$ will be supposed to be algebraically closed, so that we have $\O / P \simeq \K$, meaning places have \emph{degree $1$}.
The image of a function $f$ by \emph{the evaluation map at $P$}: $\mathcal O \rightarrow \mathcal O /P$, will be denoted by $f(P)$.

A \emph{divisor} of $F$ is a formal sum $n_1 P_1 + \cdots + n_s P_s$
where the $n_i$'s are integers and the $P_i$'s are places of $F$.
The \emph{degree} of $D = n_1 P_1 + \cdots n_s P_s$ is defined as $\sum_i n_i$ when the places $P_i$ have degree $1$.
A divisor $D = n_1 P_1 + \cdots n_s P_s$ is said to be \emph{positive} if $n_i \geq 0$ for all $i$. 
We write $D \geq D'$ to mean that $D-D'$ is positive. Since 
a function $f \in F^\times$ has non-zero valuation at only a finite number of places,
we can associate to $f$ its \emph{principal divisor}
\[
    (f)_F \eqdef \sum_P v_P(f) P,
\]
where the sum is taken over all the places of $F$. The positive and negative part 
of the above divisor are denoted $(f)_{F,0}$ and $(f)_{F,\infty}$ which means that
$(f)_F = (f)_{F,0} - (f)_{F,\infty}$. To lighten notation, when there is no ambiguity on the ambient field, the subscript $F$ is omitted and we write $(f), (f)_0, (f)_{\infty}$.
Principal divisors are known
to have degree $0$.
Finally, given a divisor $D$, the \emph{Riemann--Roch} space associated to $D$
is defined as
\[
    \L_F (D) \eqdef \{f \in F^\times ~|~ (f)+D \geq 0\} \cup \{0\}.
\]
This is a finite dimensional $\K$-vector space. We will again omit the subscript
$F$ and write $\L (D)$ when the ambient field is clear from context.
Finally, to any function field is associated
a non-negative invariant $g$ called its \emph{genus} and the Riemann--Roch Theorem asserts that
\[
    \dim_\K \L(D) \geq \deg D + 1 - g
\]
with equality when $\deg D > 2g-2$.

\Subsubsection{Products of spaces}
As mentioned in the introduction, when $S,T$ are $K$-vector subspaces of an
extension field $F$ of $K$, we denote by $ST$ the $K$-vector space generated by
products $st$ of elements of $S$ with elements of $T$. Note that this is simply the
set of sums of products of elements of $S$ and $T$. In the particular case when
$T=M$ is itself a subfield of $F$, $MS$ is simply the $M$-subspace generated by
$S$.
Also given an intermediate field $K \subset M \subset L$, then we have $M(ST) = (MS)(MT)$
and we denote this space by $MST$.
More generally, the product $ST$ of two $K$--vector spaces can be defined in any $\K$--algebra. In particular, in Section~\ref{sec:ScapL}, products of spaces in the product ring $\K^n$ are involved.

We assume the setting of Theorem~\ref{thm:main}. Precisely, let $K$ be a fixed
perfect field and $F/K$ be an extension field such that no element of
$F\setminus K$ is algebraic over $K$. Let $S$ be a finite-dimensional
$K$-subspace of $F$, such that $K\subset S$ and $F=K(S)$, i.e. $S$ generates $F$
as an extension of $K$. We denote henceforth the dimension $\dim_KS$  of $S$ by
$k$. Finally, we suppose that $S$ has combinatorial genus $0\leq\gamma\leq k-3$,
where the combinatorial genus $\gamma$ is defined to be the integer such that
\[\dim S^2=2k-1+\gamma.\]

\Subsection{The fields $F$ and $K$ in Theorem~\ref{thm:main}}

\Subsubsection{The transcendence degree of $F/K$.}
\label{subsec:transc_deg_1}
As already mentioned in the introduction, given the above assumptions on
$K,S,F$, it is already proved in \cite[Thm.~9.1]{bcz18} that the extension $F/K$
has transcendence degree $1$.

Hence, one can assume from now on that $F$ is a one variable function field over $K$.

\Subsubsection{The base field $K$ can be assumed to be algebraically closed.}
We will prove Theorem~\ref{thm:main} in the special case when $K$ is an
algebraically closed field. 
We will address the issue of the generalisation from an algebraically closed field $K$
to all perfect fields $K$ in Section~\ref{sec:perfect_field}.

From now and till the end of Section~\ref{sec:comb_proof}, $K$ \textit{is assumed to
be algebraically closed.}

\Subsection{Choosing the function $x\in S$}

\Subsubsection{Multiplicatively translating $S$}
One easily checks that, since $K\subset S$, we have $F=K(S)=K(u^{-1}S)$ for any
$u \in S\setminus \{0\}$, and that proving Theorem~\ref{thm:main} for $S$ is
equivalent to proving it for $u^{-1}S$. It will be convenient for us to thus
switch to some multiplicative translate of $S$.

Specifically, start by fixing an arbitrary place of $F$ that we shall denote by $Q_{\infty}$. 
The reason for the subscript $\infty$ will become apparent later
(see Remark~\ref{rem:Qinf}).
Consider a \emph{filtered basis} $f_0, \dots, f_{k-1}$ of $S$, 
i.e. a basis of $S$ such that $v_{Q_{\infty}}(f_0) > v_{Q_{\infty}}(f_1) > \cdots >
v_{Q_{\infty}}(f_{k-1})$, where $v_{Q_{\infty}}$ is the discrete valuation at $Q_{\infty}$ (such a basis is easily constructed). 
Replacing if need be $S$ by $f_0^{-1}S$
and setting $e_i \eqdef f_0^{-1}f_i$ for any $0 \leq i <k$, we get a filtered basis
$e_0, \dots, e_{k-1}$ with $e_0 = 1$ and
\begin{equation}
\label{eq:filtered}
  0 = v_{Q_{\infty}}(e_0) > v_{Q_{\infty}}(e_1) > \cdots > v_{Q_{\infty}}(e_{k-1}).
\end{equation}
We hereafter suppose that $S$ admits a basis $e_0=1,\ldots,e_{k-1}$ satisfying
\eqref{eq:filtered}.

\Subsubsection{The function $x$ and the divisor $D$.}
Note that for almost any place $Q$ of $F$, we have $v_Q(e_i) \geq 0$
for all $0 \leq i < k$. Consequently, the pole locus of $S$ is finite,
i.e.  there is only a finite set of places at which some elements
of $S\setminus \{0\}$ have a negative valuation. Therefore, there exists
a divisor $D$ such that $S \subset \L(D)$ and
$S \not \subset \L (D_0)$ for any $D_0 < D$. Since $1 \in S$, the divisor $D$ is positive.

We can think of $\L(D)$ as {\em the smallest Riemann-Roch space which contains $S$}.

\begin{lemma}
  There exists $x \in S$ such that its pole divisor $(x)_\infty$ satisfies
  \begin{equation}\label{eq:x}
(x)_\infty= D.
\end{equation}
\end{lemma}

\begin{proof}
  If we do not have $(x)_\infty= D$, then there is a function
  $y\in S$ that has $v_Q(y)<v_Q(x)$ for some place $Q$ in the support
  of $D$. Since $\K$ is infinite, we may add a $\K$-multiple
  $\lambda y$ of $y$ to $x$ to obtain a function $x+\lambda y$ such
  that $v_Q(x+\lambda y)=v_Q(y)$ and $v_{Q'}(x+\lambda y)\leq v_{Q'}(x)$ for
  every place $Q'\neq Q$ and supporting $D$. We then replace $x$ by $x+\lambda y$: a finite number of
  iterations of this process leads to an $x \in S$ satisfying the
  required property \eqref{eq:x}.
\end{proof}

\begin{remark}\label{rem:Qinf}
Note that the previously introduced place $Q_{\infty}$ of $F$
is a pole of $x$ and hence lies above the place \emph{at infinity} of $K(x)$.
\end{remark}

\Subsubsection{Choosing $x$ to be separating}
In Section~\ref{sec:ScapL} we will need the extension $F/\K (x)$ to be separable. 
For this reason we potentially modify the choice of $x$
while keeping Property~(\ref{eq:x}).

\begin{lemma}
  If $F/\K (x)$ is inseparable, there exists $x' \in S$ such that $F/\K(x')$
  is separable and $(x')_{\infty} = (x)_{\infty}$.
\end{lemma}

\begin{proof}
  Denote by $p$ the characteristic of $\K$. Suppose $p> 0$ (otherwise
  any extension is separable).  From \cite[Prop.~3.10.2(d)]{S09},
  since $F/\K (x)$ is inseparable then $x \in F^p$.
  Since by assumption $F = K(S)$, the space $S$ cannot be
  contained in $F^p$ and there exists $y\in S$
  such that $y \notin F^p$. Then, for any $\lambda \in \K$,
  $x + \lambda y \notin F^p$.  Since $y \in \L(D)$, we have
  $(y)_{\infty} \leq D = (x)_{\infty}$, and for all but a finite number of
  $\lambda \in \K$, 
  $(x + \lambda y)_{\infty} = (x)_{\infty}$.  Setting $x' = x+\lambda y$
  for such a $\lambda$ yields the result.
\end{proof}

Replacing $x$ by $x'$ thus constructed if need be, we may suppose that 
$(x)_{\infty} = D$ and that $F/K(x)$ is separable. Furthermore, since
$x$ has been constructed so as to minimise, among functions of $S$, its valuation at any place supporting $D$,
we have that the $Q_{\infty}$--filtered basis $e_0,\ldots ,e_{k-1}$ satisfies $v_{Q_{\infty}}(e_{k-1})=v_{Q_\infty}(x)$. We may therefore choose to have $e_{k-1}=x$ without
losing the $Q_{\infty}$--filtered structure of the $e_i$.

\Subsubsection{The zeros of $x$}\label{subsec:zeroes_of_x}
    In the sequel, we will need one last property of $x$, namely that all its
    zeros in $F$ be simple. In other words, we want the divisor $(x)_{F,0}$ to be a sum of places all occurring with multiplicity $1$. Note that $(x)_{F,0}$ is nothing but the decomposition in $F$ of the
    single place $O$ of $\K(x)$ at which $x$ vanishes (the
place $O$ will play a significant role in Section~\ref{sec:ScapL}). Now,
non-simple zeros of $x$ appear in $F$ if and only if the
place $O$ of $K(x)$ ramifies in $F$. Note that the places of $K(x)$ are the place at infinity and
    the unique zero of $(x-a)$ when $a$ ranges over $K$.
    Since in a given extension there are only finitely many ramified places, replacing if need be
    $x$ by $x-a$ for some $a$, one can suppose that the zero $O$ of $x$ in $\K(x)$ is unramified
    in $F$. Finally, note that this last change of variable (replacing $x$ by $x-a$) has no influence on the pole locus of $x$
    and does not change the field $\K(x)$. It can therefore be applied without modifying the divisor $D$ and without harming the previous properties of $x$.

\medskip

Summarising: from now on, we suppose that
\begin{align*}
   &   e_0 = 1,\quad e_{k-1} = x,\\
   &  (x)_{\infty} = D, \\
   &   \text{$F/\K(x)$ is separable,}\\
   &   \text{All the zeros of $x$ in $F$ are simple.} 
\end{align*}

\medskip

Finally, we recall that:
\begin{equation}
  \label{eq:key}
  \deg(x)_\infty = [F:\K(x)] = \deg D.
\end{equation}

\Subsection{Applying Kneser's Theorem for extension fields}
Let us recall Kneser's Theorem for extension fields~\cite{hlx02,bsz18}.
\begin{theorem}
\label{thm:kneser}
Let $E\subset M$ be fields and let $A,B$ be finite-dimensional $E$-subspaces of
$M$ such that $A\neq\{0\},B\neq\{0\}$. Then
\[
\dim_EAB \geq \dim_EA+\dim_EB-\dim_E\mathrm{St}(AB)
\]
where $\mathrm{St}(AB)=\{x\in M, xAB\subset AB\}$ is the stabiliser of $AB$ in $M$.
\end{theorem}
Note that since $AB$ is finite-dimensional over $E$, so is $\mathrm{St}(AB)$.
Since $\mathrm{St}(AB)$ is contained in $M$, it is also an integral ring and
hence it must be a field.

We shall apply Theorem~\ref{thm:kneser} to $E=K(x)$, $M=F$, and $A=B=K(x)S$.
Let $L=\mathrm{St}(K(x)S^2)$ be the stabiliser of $K(x)S^2$. We first show that $L$ is
non-trivial, i.e. $L\neq K(x)$.

\Subsubsection{$L$ is a non-trivial extension of $K(x)$}
\label{sec:nontrivial}
Consider the subspace $S+xS$ of $S^2$. We have that $S \cap xS=Kx$. Indeed, an element
$f \in S \cap xS$ is an element of $xS$ and hence satisfies
$(f)_{\infty} \geq (x)_{\infty}$. On the other hand, as an element of
$S$ it should satisfy $(f)_{\infty} \leq (x)_{\infty}$.  Therefore,
$(f)_{\infty} = (x)_{\infty}$ and then $(x^{-1}f) \geq 0$, i.e. $x^{-1}f$ has no
poles, meaning  $x^{-1}f \in \K$ and hence $f \in Kx$.
Consequently,
  \[\dim_\K(S+xS) = 2k-1.\]
  Since $\dim_KS^2=2k-1+\gamma$, there exists some $\K$--subspace $T$ of $F$ satisfying
  \begin{equation}\label{eq:decomp_S2}
    S^2=(S+xS)+T\quad \text{and}\quad
    \dim_\K T = \gamma.
  \end{equation}
  Now, consider the spaces $\K (x)S$ and
  $\K(x)S^2$. Since all $Q_{\infty}$-adic valuations of functions in $\K(x)$
  are multiples of $v_{Q_{\infty}}(x)$, we have that
  \[1,e_1,\ldots ,e_{k-2}\]
  must be linearly independent over $\K(x)$, hence
  \begin{equation}\label{eq:dimKxS}
    \dim_{\K(x)} \K (x)S=k-1.
  \end{equation}
  We also have $ \K (x)(S+xS)=\K (x)S$, which, together with
  (\ref{eq:decomp_S2}) entails
  \begin{equation}
    \label{eq:n-1+gamma}
    \dim_{\K(x)} \K (x)S^2\leq \dim_{\K(x)} \K (x) S+\dim_{\K(x)} \K (x) T\leq k-1+\gamma.
  \end{equation}
From \eqref{eq:dimKxS}, \eqref{eq:n-1+gamma} and since $\gamma\leq k-3$, we therefore have
\begin{align*}
  \dim_{\K (x)} \K (x)S^2 \leq k-1+\gamma &= \dim_{\K(x)} \K(x)S + \gamma\\
  & < 2 \dim_{\K (x)} \K (x)S - 1.
\end{align*}
Theorem~\ref{thm:kneser} therefore implies that
$L\varsupsetneq K(x)$.

\Subsubsection{The hypothesis $L=F$}
It could happen that $L=F$, i.e. that $K(x)S^2$ is stabilised by $F$, which
simply means that $K(x)S^2=F$. We now proceed to show that if this is the case, then the conclusion of Theorem~\ref{thm:main} holds. 
This is captured by the following proposition.
\begin{proposition}
If $L=F$ or, then the genus $g$ of $F$ satisfies $g\leq\gamma$ and 
 $\dim_\K \L (D) - \dim_KS \leq \gamma - g$, i.e.  the space $S$ is of
codimension at most
$\gamma-g$ inside a Riemann-Roch space.
\end{proposition}
\begin{proof}
From \eqref{eq:key} and \eqref{eq:n-1+gamma},
the hypothesis $L=F$, or equivalently $K(x)S^2 = F$, gives 
\begin{equation}
\label{eq:degD}
  {\deg D \leq k-1+\gamma.}
\end{equation}
Let us now prove that the genus $g$ of $F$
satisfies $g\leq\gamma$. By Clifford's Theorem \cite[Thm.~1.6.13]{S09}, if $\deg D\leq 2g-2$, then
we have $\dim \L(D)\leq 1+\frac 12\deg D$. Since $S\subset \L(D)$ we have
$k\leq \dim \L(D)$ so that
\[2k-2\leq\deg D,\]
which together with \eqref{eq:degD} gives $k\leq \gamma +1$ in
contradiction with the hypothesis $k\geq \gamma +3$. Therefore,
$\deg D\leq 2g-2$ cannot hold, and by the Riemann-Roch Theorem $\deg
D\geq 2g-1$ implies $\dim \L(D)=\deg D+1-g$. Hence,
\begin{align*}
  k&\leq \dim \L(D) = \deg D+1-g\\
   &\leq k-1+\gamma +1-g,
\end{align*}
where the last inequality is a consequence of \eqref{eq:degD}.
From this, we obtain
\begin{equation}\label{eq:genus_F}
  {g\leq\gamma.}
\end{equation}
Furthermore, $\dim \L(D)=\deg D+1-g$ and \eqref{eq:degD} also gives
\begin{equation}
  \label{eq:codim_in_RR}
  \dim_\K \L (D) - \dim_KS \leq \gamma - g.
\end{equation}
\end{proof}

We have therefore shown that if $L=F$ all conclusions of Theorem~\ref{thm:main} hold. 
It
remains to study the situation when $\K(x) \varsubsetneq L \varsubsetneq F$.
In the following sections we proceed to show that such a situation cannot
happen.

\Subsection{Choosing a basis of $LS$}\label{subsec:L-basis}
We conclude this section by fixing a basis of $LS$ that will play a role in the two subsequent sections.
Let us set $\ell \eqdef [L:\K (x)]$ and $\kappa = \dim_LLS$, so that we have
$\dim_{K(x)}LS=\kappa\ell$. 

\begin{lemma}
\label{lem:2kappa}
We have $(2\kappa-1)\ell\leq 2k-4$.
\end{lemma}
\begin{proof}
Since $L$ is the
stabiliser of $LS^2=K(x)S^2$, Kneser's Theorem~\ref{thm:kneser} gives us 
that $\dim_{K(x)}LS^2\geq 2\dim_{K(x)}LS - \dim_{K(x)}L=2\kappa\ell-\ell$.
Together with \eqref{eq:n-1+gamma} and $\gamma\leq k-3$ this gives us the lemma.
\end{proof}

 Let us focus on the
place $Q_{\infty}$ relative to which we have chosen a filtered basis of $S$. We have $v_{Q_{\infty}}(x)=-N$ for a certain positive integer $N$,
and the set of $Q_{\infty}$--adic valuations of the functions of $K(x)$ is therefore $N\Z$. The set
of $Q_{\infty}$--adic valuations of the functions of $L$ is therefore a subgroup
$d\Z$ for some $d|N$. Since elements of
$L$ of valuations $0,-d,-2d,\ldots ,-(N/d-1)d$ must be linearly independent over
$K(x)$, we have $N/d\leq\ell$.
Now in the interval $[-N,0]$ we have $N/d+1$ integers that are multiples of $d$,
and exactly $N/d$ representatives of every non-zero class of integers modulo
$d$. Therefore, the number of integers of $[-N,0]$ that fall into the union of
some $\kappa-1$ classes modulo $d$ is bounded from above by $(\kappa-1)\ell +1$.
But Lemma~\ref{lem:2kappa} implies that $(\kappa-1)\ell +1<k$.
Since the set of valuations at $Q_{\infty}$ of the functions of $S$ must have cardinality equal
to $k=\dim_KS$, it has to include at least $\kappa$ distinct classes modulo $d$. Let us choose
$\kappa$ elements of $S$, $s_1=1,s_2,\ldots ,s_\kappa$ that have distinct
valuations $v_{Q_{\infty}}(s_i)$ modulo $d$: these elements are necessarily linearly
independent over $L$. Summarising for future reference:

\begin{lemma}
\label{lem:LSbasis}
There exist $\kappa$ elements $s_1=1,s_2,\ldots,s_\kappa$ of $S$ whose
valuations $v_{Q_{\infty}}(s_i)$ at $Q_\infty$ are distinct modulo $d$, where
$v_{Q_{\infty}}(L)=d\Z$, and make up therefore a basis of $LS$.
\end{lemma}

 \section{The dimension of $S \cap L$}
\label{sec:ScapL}
Let $\tau$ be the integer defined as:
\begin{equation}
\label{def:tau}
    \tau \eqdef \dim_{\K(x)} LS - \dim_{\K(x)} \K(x)S.
\end{equation}
Our objective in this section is to prove the following lower bound on the dimension of $S\cap L$:
\begin{proposition}
\label{prop:ScapL}
    We have
    \[
    \dim_K(S\cap L)\geq \ell + 1-\tau \quad \text{or\ equivalently}
    \quad \dim_\K S\cap L \geq \dim_\K S - (\kappa - 1)\ell.
    \]
\end{proposition}
Whereas the previous section was mainly focused on the poles of $x$ and the space $\L(D)$ for $D=(x)_\infty$, in the present section the \textit{zeros} of $x$ will play a dominant role. Since we have taken care to ensure that every pole of every function of $S$ is also a pole of $x$, it must be that every zero of $x$ is not a pole of any function of $S$. It makes sense, therefore, to evaluate the functions of $S$ on the zeros of $x$ and study the $K$-vector space generated by these evaluations.
In the first part of this section we will prove that the subspace $S_0$ of those functions of $S$ that evaluate to blockwise constant vectors (to be defined formally below) satisfies the lower bound of Proposition~\ref{prop:ScapL}: this will be given by Proposition~\ref{prop:dimS0}. The second part of the section will then be devoted to proving that this space $S_0$ must coincide with the space $S\cap L$, which will yield Proposition~\ref{prop:ScapL}.

\Subsection{The dimension of the space $S_0$}

\Subsubsection{Decomposing places in extensions}
Given a separable extension $M'$ of a function field $M$ over $K$, and given a place $P'$ of $M'$, there is a unique place $P$ of $M$ defined as $P \eqdef P' \cap M$, called the
\emph{restriction} of $P'$ to $M$. 
According to the representation on Figure~\ref{fig:P}, we follow standard practice and
say that $P'$ is \emph{above} $P$ and that the latter is \emph{below}
the former.
\begin{figure}[h]
  \begin{center}
    \begin{tikzpicture}
      \node at (0,0) {$M$};
      \node at (0,1.3) {$M'$};
      \node at (3,0) {$P$};
      \node at (2,1.3) {$P_1'$};
      \node at (3,0.9) {$\cdots$};
      \node at (4,1.3) {$P_s'$};
      \draw[-] (0,0.3) to (0,1);
      \draw[-] (2.8,0.3) to (2,1);
      \draw[-] (3.2,0.3) to (4,1);
    \end{tikzpicture}
  \end{center}
  \caption{Places above $P$}
  \label{fig:P}
\end{figure}
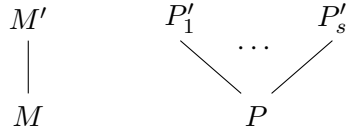
In the other direction, starting from a place $P$ of $M$, there exist finitely many places $P_1', \dots, P_s'$ of $M$ whose restriction to $M$ equals $P$. Given an element
$t \in M$ of valuation $1$ at $P$, then for any $P'$ above $P$ the \emph{ramification index of $P'$ above $P$} is defined as
\[
    e(P'|P) \eqdef v_{P'}(t)
\]
and is independent of the choice of $t$. We recall the fundamental identity (for $K$ algebraically closed)
\[
    [M' : M] = \sum_{P' \ \text{above}\ P} e(P'|P).
\]

\Subsubsection{The place $O$ and its decomposition in $L$ and $F$}
According to Section~\ref{subsec:zeroes_of_x}, the function $x$ has only simple zeros in $F$ or equivalently, the place $O \eqdef (x)_{K(x),0}$ of $\K (x)$ does not ramify in $F$.
Consequently, $O$ is also
unramified in $L$ and decomposes into places
$\Pg_1, \dots, \Pg_\ell$ of $L$. Next, these places of $L$ all split totally into places $Q_{11}, \dots, Q_{1m}, \cdots, Q_{\ell 1}, \dots, Q_{\ell m}$ of $F$. 
  Denote by $\OKx$ the local ring at $O$ of $\K(x)$,
  that is to say:
  \[
    \OKx \eqdef \left\{ \frac{f(x)}{g(x)} \in \K(x) ~\bigg|~ g(0) \neq 0
    \right\}.
  \]
  Then, denote by $\OL$ the integral closure of $\OKx$ in $L$ and by $\OF$ the
  integral closure of $\OKx$ in $F$. The rings $\OL$ and $\OF$ are
  semi-local (i.e. have a finite number of maximal
  ideals) and, from \cite[Cor.~3.3.5]{S09}, they satisfy
  \[
    \OL = \bigcap_{1 \leq i \leq \ell} \O_{\Pg_i}
    \quad \text{and} \quad
    \OF = \bigcap_{\substack{1 \leq i \leq \ell \\{1 \leq j \leq m}}} \O_{\Qg_{ij}}.
  \]
 Moreover, we have
  \[
    \OKx / x \OKx \simeq \K, \qquad
    \OL / x \OL \simeq \K^\ell \quad \text{and} \quad
    \OF / x \OF \simeq \K^n
  \]
  where $n\eqdef [F:K(x)]$ which is equal, from \eqref{eq:key}, to $\deg (x)_{F,\infty}$ and to $\deg (x)_{F,0}$.
  Finally, $\OKx$ is a discrete valuation ring, and, from
  \cite[Prop.~3.2.10]{S09} the semi--local rings $\OL$ and $\OF$ are
  principal ideal domains.
  
The situation is summarised in 
Figure~\ref{fig:places}.

\begin{figure}[h]
  \begin{center}
    \begin{tikzpicture}
      \node at (1,3) {$F$};
      \node at (3,3) {$\OF$};
      \node at (1,1.5) {\(L\)};
      \node at (3,1.5) {$\OL$};
      \node at (1,0) {\(\K (x)\)};
      \draw[-] (1,.3) to (1,1.2) ;
      \draw[-] (1,1.8) to (1,2.7) ;
      \draw[-] (1.1,1.8) .. controls (1.2,2.1) and (1.2,2.4) .. (1.1,2.7);
      \node at (1.4,2.25) {$\scriptstyle{m}$};
      \draw[-] (1.1,.3) .. controls (1.2,.6) and (1.2,.9) .. (1.1,1.2);
      \node at (1.4,.625) {$\scriptstyle{\ell}$};
      \draw[-] (.9,.3) .. controls (.6,.6) and (.6,2.4) .. (.9,2.7);
      \node at (0.1,1.5) {$\scriptstyle{n = \ell m}$};
      \node at (3,0) {\(\OKx\)};
      
      \node at (8,0) {\(O\)};
      \node at (6,1.5) {\(\Pg_1\)};
      \node at (8,1.5) {\(\cdots\)};
      \node at (10,1.5) {\(\Pg_\ell\)};
      \node at (5,3) {\(\Qg_{11}\)};
      \node at (6,3) {\(\cdots\)};
      \node at (7,3) {\(\Qg_{1m}\)};
      \node at (8,3) {\(\cdots\)};
      \node at (9,3) {\(\Qg_{\ell 1}\)};
      \node at (10,3) {\(\cdots\)};
      \node at (11,3) {\(\Qg_{\ell m}\)};
      \draw[-]  (3,0.3) to (3,1.2);
      \draw[-]  (3,1.8) to (3,2.7);
      \draw[-]  (1.45,0) to (2.35,0);
      \draw[-]  (1.35,1.5) to (2.45,1.5);
      \draw[-]  (1.35,3) to (2.45,3);

      \draw[-] (7.8,.2) to (6.2,1.2);
      \draw[-] (8.2,.2) to (9.8,1.2);
      \draw[-] (5.8,1.8) to (5.2,2.7);
      \draw[-] (6.2,1.8) to (6.8,2.7);
      \draw[-] (9.8,1.8) to (9.2,2.7);
      \draw[-] (10.2,1.8) to (10.8,2.7);
    \end{tikzpicture}
  \end{center}
\caption{Places above $O$.}
\label{fig:places}
\end{figure}
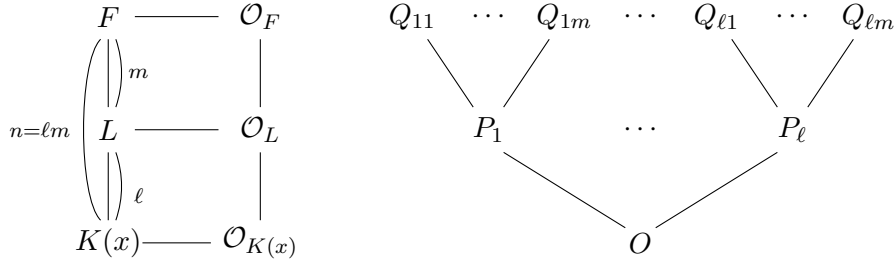

\Subsubsection{The space of evaluations}
Note that when regarded as belonging to a subspace of $\OF / x\OF \simeq \K^n$, the
elements of $\OKx/x\OKx$ are constant vectors and those of $\OL/x\OL$ are
\textit{blockwise} constant vectors. That is to say, vectors indexed by the
places $\Qg_{ij}$ whose entries are equal on places $\Qg_{i_0 j}$ for
$j \in \{1, \dots, m\}$ and a fixed $i_0$. Indeed, they are the vectors of evaluations
of functions of $\OL$ at places $\Qg_{ij}$ and hence take the same value
at places lying above a given place of $L$.

We now focus on the functions of $S$ that have blockwise constant vectors in the sense just introduced.
Recall that the integer $\tau$ is defined by \eqref{def:tau}.

\begin{proposition}\label{prop:dimS0}
  Let
    $S_0 \eqdef \{f \in S\ ~|~ \forall i \in \{1, \dots, \ell\},\
    f(\Qg_{i1}) = \cdots = f(\Qg_{im})\}$.
  In words, let $S_0$ be the subspace of functions of $S$ whose evaluation vectors
are blockwise constant. The dimension of $S_0$ satisfies:
  \[
    \dim_{\K} S_0 \geq \ell + 1 - \tau.
  \]
\end{proposition}

We will prove Proposition~\ref{prop:dimS0} in several steps. 
We will focus on
the evaluation map itself
\begin{equation}\label{eq:pi_map}
  \pi : \map{\OF}{\K^n}{f}{(f(\Qg_{11}), \dots, f(\Qg_{\ell m})),}
\end{equation}
which, by the Chinese remainder theorem, is equivalent to the canonical map
\begin{equation}\label{eq:quotient_map}
  \OF \longrightarrow \OF/x\OF.
\end{equation}
In the sequel, for an element $a$ or a subspace $A$ of $\OF$, we
denote respectively by $\bar a$ and $\overline A$ their images by the map $\pi$ of
(\ref{eq:pi_map}), or equivalently by the map (\ref{eq:quotient_map}).

Our objective is to evaluate the dimension of $\barSz$ from which the
dimension of $S_0$ can be easily be deduced as explained by the two
following statements. Notice first that $\L(D)\subset\OF$, so that
$\pi$ is well-defined on $\L(D)$.

\begin{lemma}\label{lem:kernel_ev}
  The kernel of the map $\pi_{|\L(D)} : \L(D) \longrightarrow \K^n$ is the
  space $\L (D - \sum_{i,j} Q_{ij})$, which equals the
  one--dimensional $K$--space spanned by $x$.
\end{lemma}

\begin{proof}
  Let $f \in \L(D)\setminus \{0\}$ such that $\pi(f) = 0$. By \eqref{eq:pi_map}
  $f$ vanishes at all the places $Q_{ij}$ and hence $(f) \geq  \sum_{i,j} Q_{ij} -D$.
  Recall that $D$ was defined as $D=(x)_\infty$ and since the zeros of $x$ are the $Q_{ij}$, we have
\[
    (x) = (x)_0-(x)_\infty = \sum_{i,j} {Q_{ij}} - D.
  \]
  Thus, $(f) \geq (x)$ and, since principal divisors have
  degree $0$, we deduce that $(f) = (x)$ or equivalently $(f/x) = 0$ and hence that $f$ and $x$
  are $\K$--collinear.
\end{proof}

\begin{corollary}\label{cor:dimBarS}
  The spaces $\barSz$ and $\barS$ have respective $\K$--dimensions
  $\dim_\K S_0 -1$ and $\dim_\K S - 1$.
\end{corollary}

\begin{proof}
  Both spaces are contained in $\L(D)$ and contain $x$.
\end{proof}

From now on, our objective will be to derive a lower bound for the dimension of
$\barSz$. We proceed as follows. As already mentioned, $\barOL$, when regarded as
a subspace of $K^n$ is the subspace of blockwise constant vectors. Therefore, by definition of
$S_0$,
\begin{equation}\label{eq:barS0_as_an_intersection}
    \barSz = \barS \cap \barOL.
\end{equation}
Moreover, since both $\barS$ and $\barOL$ contain $1$, they are both contained in $\barOLS = \barOL \barS$ where the former product is taken in the algebra $\OF/x\OF$ or equivalently in $K^n$
equipped with the componentwise product. 
Consequently, $\barOL + \barS \subseteq \barOLS$ which entails
\[
    \dim_K \barOL \barS \geq \dim_K \barOL + \dim_K \barS - \dim_K \barOL \cap \barS.
\]
Using \eqref{eq:barS0_as_an_intersection}, the above inequality yields
\begin{equation}\label{eq:minor_dim_barS0}
    \dim_K \barSz  \geq \dim_K \barOL + \dim_K \barS - \dim_K \barOL \barS.
\end{equation}

Therefore, we need to estimate the dimensions of $\barOL$ and $\barOLS$.

\begin{lemma}\label{lem:dimBarOL}
  $\dim_{\K} \barOL = \ell$
\end{lemma}

\begin{proof}
  First observe that $\barOL$ is isomorphic to
  $\OL / (\OL \cap x \OF) = \OL / x\OL$. From \cite[Cor.~3.3.5]{S09},
  $\OL$ is an $\OKx$--free module of rank $\ell$. Then, by Nakayama's
  Lemma \cite[Prop.~2.8]{AM69}, $\OL / x\OL$ is a vector space over $\OKx / x\OKx \simeq \K$
  of dimension~$\ell$.
\end{proof}

\begin{remark}
The previous lemma can be made more explicit: indeed, since $\OL / x\OL$, when regarded as a
subspace of $\K^n$, is made up of blockwise constant vectors, it follows from $\dim_K\barOL = \ell$ that $\barOL$
is exactly the subspace of blockwise constant vectors of $\K^n$, meaning that it
is spanned by the vectors
  \begin{eqnarray*}
    \bar v_1 & = & (1, \dots, 1, 0, \dots, 0, \cdots, 0, \dots, 0);\\
    \bar v_2 & = & (0, \dots, 0, 1, \dots, 1, \cdots, 0, \dots, 0);\\
        & \vdots & \\
    \bar v_\ell & = & (0, \dots, 0, 0, \dots, 0, \cdots, 1, \dots, 1).
  \end{eqnarray*}
\end{remark}

The estimate of the dimension of $\barOLS$ will rest upon the following statement.

  \begin{proposition}\label{prop:rankOLS}
    The module $\OL S$ is a free $\OKx$--module whose rank satisfies
    \[
      \mathrm{Rk}_{\OKx} \OL S \leq \dim_{\K(x)}LS = k-1+\tau,
    \]
    where $\tau$ is defined by \eqref{def:tau}. \end{proposition}

  \begin{proof}
    Since $S$ is finite dimensional over $K$ and $\OL$ is free and finitely
    generated over $\OKx$ (\cite[Cor.~3.3.5]{S09}), we deduce that
    $\OL S$ is a finitely generated module over $\OKx$ which is a
    principal ideal domain. Hence, by the structure theorem of
    finitely generated modules over a principal ideal domain, $\OL S$
    is the direct sum of a free $\OKx$-module and a torsion module.
    However, since $\OL S \subseteq F$ and $F$ is a field, $\OL S$ is torsion--free over $\OKx$
    and hence is a free finitely generated $\OKx$--module.

    Next, let $r$ be the $\OKx$--rank of the module $\OL S$ and let
    $t_1, \dots, t_r$ be an $\OKx$--basis of $\OL S$, we claim that
    $t_1, \dots, t_r$ are $\K (x)$--independent. Indeed, suppose that
    \begin{equation}\label{eq:linear_comb_in_OLS}
      f_1 t_1 + \cdots + f_r t_r = 0
    \end{equation}
    for some $f_1, \dots, f_r \in \K (x)$. Then, since $\OKx$ is a
    discrete valuation ring of $\K (x)$ with local parameter $x$, for
    some integer $j$ the functions $x^j f_1, \dots, x^j f_r$ all lie
    in $\OKx$. Therefore, multiplying (\ref{eq:linear_comb_in_OLS}) by a suitable $x^j$ we get an $\OKx$-linear relation between the $t_i$'s which
    entails that the $x^j f_i$'s are all zero and hence so are the
    $f_i$'s. Thus, the $t_i$'s are $\K (x)$--independent in $LS$ and hence
    $r = \textrm{Rk}_{\OKx} \OL S   \leq \dim_{\K (x)} LS$.
  \end{proof}

    \begin{corollary}\label{cor:dim_barOLS}
      We have $\dim_K \barOLS \leq k-1+\tau$.
    \end{corollary}
    
    \begin{proof}
        The evaluation map
    \(
      \OL S \longrightarrow \OF / x\OF
    \)
    has $x \OL S$ in its kernel, inducing a $K$--linear map
    \[
      \OL S/ x\OL S \longrightarrow \OF / x\OF
    \]
    whose image is nothing but $\barOLS$.
    Thus, by the rank nullity theorem,
    \[
      \dim_{\K} \barOLS \leq \dim_{\K} \OL S / x \OL S.
    \]
    From Nakayama's Lemma \cite[Prop.~2.8]{AM69}, the dimension of
    the right--hand side is the $\OKx$--rank of $\OL S$, which, from
    Proposition~\ref{prop:rankOLS} yields
    \begin{equation}\label{eq:dimBarOLS}
      \dim_{\K} \barOLS \leq k-1+ \tau.
    \end{equation}  
    \end{proof}

  We now have all we need to prove Proposition~\ref{prop:dimS0}.

  \begin{proof}[Proof of Proposition~\ref{prop:dimS0}]
  Corollary~\ref{cor:dimBarS}, Lemma~\ref{lem:dimBarOL} and Corollary~\ref{cor:dim_barOLS} yield estimates for the dimensions of $\barOL$, $\barS$ and $\barOLS$ respectively.
  Substituting them in \eqref{eq:minor_dim_barS0} yields
  \[
    \dim_K \barSz \geq \ell - \tau.
  \]
    Corollary~\ref{cor:dimBarS} permits to conclude that
    \[
      \dim_{\K} S_0 \geq \ell + 1 - \tau.
    \]
  \end{proof}

\Subsection{The space $S_0$ equals $S \cap L$} 
Since the evaluation of elements of $S$ on the places $Q_{ij}$ is well-defined,
we have that $S\cap L\subset \OL$, and since we have seen that functions of $\OL$ have
blockwise constant evaluations, we have that $S\cap L\subset S_0$.
Our present objective is now to prove that the converse statement holds:
\begin{proposition}\label{prop:S_0=SinterL}
  $S_0  = S \cap L$.
\end{proposition}
Together with Proposition~\ref{prop:dimS0}, Proposition~\ref{prop:S_0=SinterL}
will therefore yield Proposition~\ref{prop:ScapL}.

Let us briefly recall the setting described in Section~\ref{subsec:L-basis}. There is a place $\Qinf$ above the place
at infinity of $\K (x)$ for which $S$ has a filtered basis with
decreasing valuations from $0$ to $-N$ for some $N>0$.  Denote by
$\Pinf$ the restriction of $\Qinf$ to $L$. The elements of $L$ have a
$\Qinf$-adic valuation in $d\Z$ for some integer $d$ or
equivalently, the ramification index $e(\Qinf | \Pinf)$ equals $d$.

To prove Proposition~\ref{prop:S_0=SinterL}, we suppose the existence of $u \in S_0 \setminus S\cap L$ and work towards a contradiction.
We will use Galois theory and since $F/L$ has no reason
to be a Galois extension, we introduce $\Fgal$ to be the Galois closure of
$F/L$, that is to say the smallest Galois extension of $L$ containing $F$.
Since $F/\K(x)$ is separable, then so is $F/L$ and hence such a
closure exists and is a finite extension of $F$. We denote by $G$
the Galois group of the extension $F'/L$.

Let us observe some first properties of the Galois action of $\Gal$ on
the objects we are dealing with. Recall first a well--known fact
 (see for instance \cite[Lem.~3.5.2(a)]{S09}). Given $f \in \Fgal$, $\Qg$ a
place of $\Fgal$ and $\sigma \in \Gal$, then
\begin{equation}\label{eq:Galois_valuation}
  v_{\sigma(\Qg)}(\sigma(f)) = v_{\Qg}(f).
\end{equation}
Finally, we denote by $\Dgal$ the conorm of $D$ in $\Fgal$,  i.e.
the pole locus of $x$ when regarded as an element of $\Fgal$. In
addition, let $\Qgalinf$ be a place of $\Fgal$ above $\Qinf$, which is a place
in the support of $\Dgal$ (see Figure~\ref{fig:infty}).

\begin{lemma}\label{lem:invariant_div}
  The divisor $\Dgal$ is $\Gal$--invariant and $\L(\Dgal)$ is
  a linear representation of~$\Gal$.
\end{lemma}

\begin{proof}
  Since $\Dgal$ is the pole locus of $x$ and $x \in L$
  is $\Gal$--invariant, so is its pole locus $\Dgal$.
  Now, if for some
  $f \in \Fgal$ we have $(f) \geq - \Dgal$, then,
  from~(\ref{eq:Galois_valuation}), we deduce that for any
  $\sigma \in \Gal$, we have
  $(\sigma (f)) \geq -\sigma (\Dgal) = -\Dgal$. Thus, $\sigma$ leaves
  $\L (\Dgal)$ globally invariant and the action is linear since $\sigma$
  is a field automorphism of $\Fgal$.
\end{proof}

From now on let $u$ be a fixed, hypothetical, element $u \in S_0 \setminus S \cap L$.

\begin{lemma}
  There
  exists an additive character $\chi_u : \Gal \rightarrow \K$,
  satisfying
  \[
    \forall \sigma \in \Gal,\quad \sigma(u) = u + \chi_u(\sigma)x.
  \]
\end{lemma}

\begin{proof}
  Let $\sigma \in \Gal$. By definition of $S_0$, for any
  $i \in \{1, \dots, \ell\}$, $u$ has constant evaluation at the places $Q_{ij}$ that lie above $P_i$ in $F$.
  Since $u\in F$, we also have that $u$, viewed as a function of $F'$, has constant evaluation at all the places of $F'$ that lie above any given $Q_{ij}$. Therefore, $u$ has constant evaluation at all the places of $F'$ that lie above a given place $P_i$ of $L$.
Since $\sigma$ acts by permutation
  on the places of $F'$ above $\Pg_i$, the evaluations of $\sigma(u)$ at the
  places above $\Pg_i$ are exactly the same as the evaluations of $u$.
  Equivalently, $\sigma(u) - u$ vanishes at any place of $F'$ above the
  $\Pg_i$'s. In other words, $\sigma(u)- u $ vanishes at any zero of $x$.

  Since $u \in \L_F(D)$, we have $u \in \L_{F'}(\Dgal)$ when
  regarded as an element of $\Fgal$. Therefore, from
  Lemma~\ref{lem:invariant_div}, $\sigma (u) \in \L_{F'}(\Dgal)$ and
  $\sigma(u) - u\in\L_{F'}(\Dgal)$ also. Since $\Dgal$ is nothing but the pole locus of $x$
  when regarded as an element of $\Fgal$, we deduce that
  \[
    (\sigma (u) - u)_{\Fgal} \geq (x)_{\Fgal}.
  \]
  Since principal divisors have degree zero, the
  above inequality is an equality. Thus $\sigma (u) - u = \lambda x$
  for some $\lambda \in \K$ which we now denote by $\chi_u(\sigma)$.

  There remains to prove that the map
  $\map{\Gal}{\K}{\sigma}{\chi_u(\sigma)}$ is an additive character.
  Let $\mu \in \Gal$,
  \begin{align*}
    \mu \circ \sigma (x) & = \mu (u + \chi_u(\sigma)x) \\
                         &= \mu(u) + \mu(\chi_u(\sigma)x) \\
                         & = u + \chi_u(\mu)x + \mu (\chi_u(\sigma) x)
  \end{align*}
  Now, $\chi_u(\sigma)x \in L$ is fixed by any element of $\Gal$, which yields
  \[
    \mu \circ \sigma (x) = u + (\chi_u(\mu)+\chi_u(\sigma)) x.
  \]
  In summary, $\chi_u(\mu \circ \sigma) = \chi_u(\mu) + \chi_u(\sigma)$
  which means that $\chi_u$ is a group morphism.
\end{proof}

\begin{remark}
  If $K$ has characteristic zero, the proof of Proposition~\ref{prop:S_0=SinterL} stops here since
  the additive character $\chi_u$ would be a morphism from a finite group
  into a torsion-free group, hence must be trivial, implying $u\in L$.
\end{remark}

\begin{lemma}\label{lem:GalL(u)}
  The extension $L(u)/L$ is Galois, with Galois group $\Gal / \ker \chi_u$.
\end{lemma}

\begin{proof}
By Galois correspondence, $L(u)$ is the subfield ${F'}^H$ of elements of $F'$ fixed by some
subgroup $H$ of $G$. Elements of $L(u)$ are left fixed by $\ker \chi_u$ and hence
$\ker \chi_u \subseteq H$. Conversely, any $\sigma \in H$ fixes u and thus is in $\ker \chi_u$.
Thus $H = \ker \chi_u$ which is a normal subgroup of $G$ and by Galois
  correspondence its invariant subfield $L(u)$ is Galois over
  $L$ with Galois group $\Gal / \ker \chi_u$.
\end{proof}

Now, let us consider the places at infinity. Recall that we chose a
pole $\Qinf$ of $x$ in $F$ at which $x$, as explained in Section~\ref{subsec:L-basis}, has valuation $-N$. We also chose a
place $\Qgalinf$ above $\Qinf$. Next, $\Pinf$ is the restriction of $\Qinf$
(or of $\Qgalinf$) to $L$. Finally, we denote by $\Uinf$ the restriction of  $\Qinf$
to $L(u)$. The situation is summarised in Figure~\ref{fig:infty}.

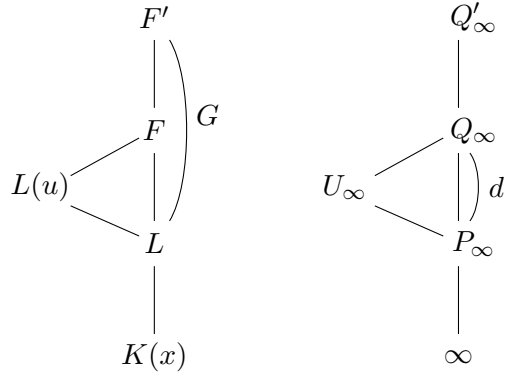
\begin{figure}
  \begin{center}
    \begin{tikzpicture}
      \node at (1,4.5) {$\Fgal$};
      \node at (1,3) {$F$};
      \node at (1,1.5) {\(L\)};
      \node at (-.5,2.25) {\(L(u)\)};
      \node at (1,0) {\(\K(x)\)};
      \node at (1.7,3.2) {\(G\)};
      \draw[-] (1,1.8) to (1,2.7) ;
      \draw[-] (1,3.3) to (1,4.2) ;
      \draw[-] (-.1,2) to (.8,1.6) ;
      \draw[-] (-.1,2.4) to (.8,2.9) ;
      \draw[-] (1,.3) to (1,1.2) ;
      \draw[-] (1.2,1.8) .. controls (1.5,2.3) and (1.5,3.7) .. (1.2,4.2);

      \node at (5.2,4.5) {$\Qgalinf$};
      \node at (5.2,3) {$\Qinf$};
      \node at (5.2,1.5) {\(\Pinf\)};
      \node at (3.5,2.25) {\(\Uinf\)};
      \node at (5,0) {\(\infty\)};
      \node at (5.5,2.25) {\(d\)};
      \draw[-] (5,1.8) to (5,2.7) ;
      \draw[-] (5,3.3) to (5,4.2) ;
      \draw[-] (3.90,2) to (4.8,1.6) ;
      \draw[-] (3.90,2.4) to (4.8,2.9) ;
      \draw[-] (5,.3) to (5,1.2) ;
      \draw[-] (5.15,1.8) .. controls (5.3,2) and (5.3,2.5) .. (5.15,2.7);
    \end{tikzpicture}
  \end{center}
  \caption{Places above $\infty$ and subfields above $K(x)$}
  \label{fig:infty}
\end{figure}
  Recall that elements of $L$ have $\Qinf$--adic valuation in $d\Z$
  which means that the ramification index $e(\Qinf | \Pinf)$ equals $d$
  as shown on the figure.
  Recall also that $u \in \L(D)$, which entails that
  $v_{\Qinf}(u) \geq v_{\Qinf}(x)$. When regarded as elements of
  $L(u)$, we deduce that $v_{\Uinf}(u) \geq v_{\Uinf}(x)$.
  Note that, 
if the inequality is not strict, then for some choice of $\delta \in \K$, we have $v_{\Uinf}(u+\delta x) > v_{\Uinf}(x)$.
  Moreover, when replacing $u$ by $w := u + \delta x$, the character $\chi_{w}$
  is unchanged. Indeed, for $\sigma \in \Gal$,
  \begin{align*}
    \sigma (w) & = \sigma (u + \delta x) = \sigma (u) + \sigma (\underbrace{\delta x}_{\in L}) = \sigma (u) + \delta x \\
                & = u + \delta x + \chi_u(\sigma)x\\
                & = w + \chi_u(\sigma) x.
  \end{align*}
  In short, $\chi_u = \chi_{w}$. Thus, replacing $u$ by $w$ if need be, from now on one can assume that
  \begin{equation}\label{eq:inequality_valuations}
    v_{\Uinf}(u) > v_{\Uinf}(x).
  \end{equation}

  \begin{lemma}\label{lem:unramified}
    The place $\Uinf$ is unramified over $\Pinf$.
  \end{lemma}

  \begin{proof}
    From \cite[Thm.~3.8.2(a)]{S09}, the decomposition group
    \[
      D_{\Uinf|\Pinf} \eqdef \{\sigma \in \text{Gal}(L(u)/L) ~|~
      \sigma(\Uinf) = \Uinf \}
    \] is non-trivial if and only if the place is ramified.
    Let $\sigma$ be an element of the decomposition group.
    Then, from (\ref{eq:Galois_valuation}),
    \begin{align*}
    \nonumber  v_{\Uinf} (u) & = v_{\sigma (\Uinf)}(\sigma (u)) \\
    \nonumber               & = v_{\Uinf}(\sigma (u)).
    \end{align*}
    Next, by the very definition of the decomposition group,
    \begin{equation}\label{eq:vRinf}  
    v_{\Uinf} (u)   = v_{\Uinf} (u + \chi_u(\sigma)x),
    \end{equation}
    If $\chi_u(\sigma) \neq 0$, then, from (\ref{eq:inequality_valuations}),
    the right--hand side of \eqref{eq:vRinf} is $v_{\Uinf}(\chi_u(\sigma) x)$, but then equality with the left-hand side
    immediately contradicts (\ref{eq:inequality_valuations}) itself.

    Thus, $\chi_u (\sigma) = 0$.  But from
    Corollary~\ref{lem:GalL(u)}, the Galois group of $L(u)/L$ is
    $\Gal / \ker \chi_u$. Hence, for any 
    $\sigma \in \text{Gal}(L(u)/L) \setminus \{\text{id}\}$, $\chi_u(\sigma) \neq 0$.  Thus,
    $\sigma = \text{id}$ and the decomposition group is trivial. Hence
    the result.
  \end{proof}

  Consequently, since
  \[
    e(\Qinf | \Pinf) = e(\Qinf | \Uinf) e(\Uinf | \Pinf) = d,
  \]
  Lemma~\ref{lem:unramified} yields
  \[
    e(\Qinf | \Uinf) = d.
  \]
  
  \begin{corollary}\label{cor:valLuindZ}
    The $\Qinf$--adic valuations of elements of $L(u)$ are in $d\Z$.
  \end{corollary}

  Now we have all we need to finally prove Proposition~\ref{prop:S_0=SinterL}.

  \begin{proof}[Proof of Proposition~\ref{prop:S_0=SinterL}]
    Let $u \in S_0 \setminus S \cap L$. Using the $L$--basis of $LS$ guaranteed by Lemma~\ref{lem:LSbasis},
    and since $u \in S$, there exist $\ell_1, \dots, \ell_\kappa \in L$ such that
    \begin{align*}
      u & = \ell_1 + \ell_2 s_2 + \cdots + \ell_{\kappa} s_{\kappa}\\ 
      u - \ell_1 & =  \ell_2 s_2 + \cdots + \ell_{\kappa} s_{\kappa}.
    \end{align*}
    The left--hand side is in $L(u)$ and, from
    Corollary~\ref{cor:valLuindZ}, has its $\Qinf$--adic valuation in
    $d\Z$.  The terms $\ell_i s_i$ on the right--hand side have distinct valuations since, from Lemma~\ref{lem:LSbasis}, they must be distinct
    modulo $d$. Thus, the
    $\Qinf$--adic valuation of the right--hand side is that of one of
    its terms, but all of them have a valuation non divisible by $d$.
    Thus, the only possibility would be that both sides are zero, yielding $u = \ell_1 \in L$. A contradiction.
  \end{proof}

 \section{From Proposition~\ref{prop:ScapL} to a proof of Theorem~\ref{thm:main}:
combinatorial arguments}
\label{sec:comb_proof}
\Subsection{Laying the groundwork}
\label{sec:groundwork}
Notice that the hypothesis of Theorem~\ref{thm:main}, namely $\dim_KS^2\leq
2k-1+\gamma$ and $\gamma\leq k-3$, where $k=\dim_KS$, implies
\begin{equation}
\label{eq:freiman}
\dim_KS^2\leq 3k-4
\end{equation}
which we will refer to as the \textit{Freiman hypothesis} (hence the name
``$3k-4$ Theorem''). 

Recall that $\kappa$ denotes $\kappa=\dim_LLS$ so that we have
$\dim_{K(x)}LS=\kappa\ell$, and recall from definition~\eqref{def:tau}
\begin{equation}\label{eq:taubis}
\tau=\dim_{K(x)}LS-\dim_{K(x)}K(x)S=\kappa\ell -k+1.
\end{equation}

From Theorem~\ref{thm:kneser} (Kneser for extension fields) we have
\begin{equation}
\label{eq:2kappa-1}
\dim_LLS^2\geq 2\kappa -1
\end{equation}
since the stabiliser of $LS^2$ is $L$, by definition of $L$.
If $\kappa=1$ we have $LS^2=LS=L$ and, since $S$ generates $F$ over $\K$, we have $L=F$.
That situation was dealt with in
Section~\ref{sec:L=F}, so we are now in the setting $\kappa >1$. Our goal is to
derive a contradiction with the Freiman hypothesis \eqref{eq:freiman}, thus
proving that we can only have $\kappa=1$ which will prove
Theorem~\ref{thm:main}. 

The following lemma will be useful:
\begin{lemma}
\label{lem:W}
Let $W$ be a $K$-subspace of $S^2$ such that $\dim_LLW=w\leq 2\kappa -1$. Then
\[\dim_K S^2\geq \dim_KW+(2\kappa-1-w)\ell.\]
\end{lemma}
\begin{proof}
Let $W'$ be a complementary subspace of $W$ in $S^2$, i.e. $S^2 = W \oplus W'$.
Thus, $K(x)S^2 = K(x)W + K(x) W'$ and
\begin{equation}\label{eq:dimW'}
    \dim_K W' \geq \dim_{K(x)} K(x)W' \geq \dim_{K(x)} K(x)S^2 - \dim_{K(x)} K(x)W.
\end{equation}
Since $K(x)S^2 = LS^2$, from \eqref{eq:2kappa-1} we have $\dim_{K(x)} K(x)S^2 \geq (2 \kappa - 1)\ell$. In addition, since $\dim_{L} LW = w$, we have $\dim_{K(x)} LW = w\ell $ and hence
$\dim_{K(x)} K(x)W \leq w\ell$.
Substituting the previous estimates for $\dim_{K(x)} K(x)S^2$ and $\dim_{K(x)} K(x)W$ in \eqref{eq:dimW'}, we get $\dim_K W' \geq (2\kappa - 1 - w) \ell$ which concludes the proof.
\end{proof}
Applying Lemma~\ref{lem:W} to $W=S^2\cap LS$ gives therefore
\begin{equation}\label{eq:(kappa-1)ell}
\dim_KS^2\geq \dim_K(S^2\cap LS) + (\kappa-1)\ell.
\end{equation}
We have $(S^2\cap LS)\supset (S+xS)$, and recall from
Section~\ref{sec:nontrivial} that $\dim_K(S+xS)=2k-1$. Inequality
\eqref{eq:(kappa-1)ell} together with \eqref{eq:taubis} therefore give us
\begin{equation}
\label{eq:tobeimproved}
\dim_KS^2\geq 2k-1 + (k+\tau-1) - \ell =3k-2 -(\ell -\tau).
\end{equation}
This does not contradict the Freiman hypothesis \eqref{eq:freiman}, but it is
not too far from it: to obtain a contradiction, what we need to do is to improve
the lower bound \eqref{eq:tobeimproved} by $\ell-\tau-1$.

We will improve upon \eqref{eq:tobeimproved} 
in two ways. We will use valuation arguments to obtain that $(S^2\cap LS)$ is
potentially strictly larger than $S+xS$, yielding a larger term than $2k-1$ in
\eqref{eq:tobeimproved}. We will also need to improve upon the $(\kappa-1)\ell$
estimate in \eqref{eq:(kappa-1)ell}. For this we will use the {\em
Cauchy-Davenport inequality}:
\begin{lemma}
\label{lem:CD}
If $U$ and $V$ are $K$-vector subspaces of $F$ of finite dimension, then 
\[
\dim_KUV\geq \dim_KU+\dim_KV-1.
\]
\end{lemma}
\begin{proof}
This is a direct consequence of Theorem~\ref{thm:kneser}, since $F$ does not
contain finite-dimensional extensions of $K$.
\end{proof}
The strategy will be to look for $K$-subspaces $U,V$ of $S$ for which
Lemma~\ref{lem:CD} will yield a subspace $UV$ of $S^2$ such that  
$\dim_KUV > \ell\dim_LLUV$. Care will be needed to ensure that these ``extra
$K$-dimensions'' add up to a sufficiently large improvement of the lower bound
\eqref{eq:tobeimproved}.

Proposition~\ref{prop:S_0=SinterL} allows us to forget how $S_0$ was originally defined in Section~\ref{sec:ScapL}, and we may now
use the notation $S_0$ simply to refer to $S\cap L$.

Finally, recall from Lemma~\ref{lem:LSbasis} that there exist $\kappa$ elements $s_1=1,s_2,\ldots,s_\kappa$ of $S$
whose valuations at $Q_\infty$ lie in distinct classes modulo $d$. We may suppose
additionally that each $v_{Q_\infty}(s_i)$ is maximal within
its class modulo $d$. Let these elements $s_i$ of $S$ remain fixed
throughout Section~\ref{sec:comb_proof}. 

\Subsection{The decomposition of $LS$ and the induced subspaces of $S$}

The following technical lemma gives us the existence of a special decomposition of the
space $LS$ that will play an important role in bounding from below the dimension
of $S^2$.
We use bold uppercase letters to denote
vector spaces over $L$, while $K$-subspaces of $S$ are denoted with
regular uppercase letters.

\begin{lemma}
\label{lem:ABC}
There exists an element,
$s \in \{s_2, s_3, \ldots, s_{\kappa}\}$,
and three $L$-subspaces $\mathbf{A}, \mathbf{B}, \mathbf{C}$
of respective $L$-dimensions $a\geq 1,b\geq 1$ and $c$, such that $L \subset \mathbf{A}$ 
and
\[
LS = \mathbf{A} \oplus \mathbf{B} \oplus \mathbf{C},
\]
that satisfy the following properties:
\begin{enumerate}
\item\label{item:cond1} $\mathbf{A}s \subset LS$;
\item\label{item:cond2} $LS \cap \mathbf{B}s = \{0\}$;
\item\label{item:cond3} The subspace $A = S \cap \mathbf{A}$
has a dimension $\dim_K A$ satisfying \[\dim_K A \leq \dim_K S_0 + (a-1)\ell,\] 
\item\label{item:cond4} The subspace $A^{+} = S \cap (\mathbf{A} + \mathbf{B})$
has a dimension satisfying \[\dim_K A^{+} \geq \dim_K S - c\ell.\] 
\end{enumerate}
\end{lemma}

\begin{proof}
  We will iteratively construct a sequence of triples
  $\mathbf{A}_i, \mathbf{B}_i, \mathbf{C}_i$ satisfying all conditions
  of the lemma except possibly condition~\ref{item:cond3}, and show that we must
  eventually arrive at a triple satisfying all conditions, including condition \ref{item:cond3}. 
  The associated element $s$ changes at each iteration,
  though we will keep denoting it by $s$ to lighten notation: this should not generate undue confusion. 
  The $L$--dimensions of the
  spaces $\mathbf{A}_i, \mathbf{B}_i, \mathbf{C}_i$ will be denoted as
  $a_i, b_i, c_i$ respectively, and the associated subspaces of $S$
  will be denoted as $A_i$ and $A_i^+$.

We start by choosing any $s \in\{s_2,\ldots s_\kappa\}$. 
We then define the $L$-subspace $\mathbf{A}_0$ of $LS$ consisting of those $\mathbf{y} \in LS$ such that
$\mathbf{y}s \in LS$. Clearly $L\subset\mathbf{A}_0$; 
note also that $LSs \neq LS$, otherwise
$LS$ would be stabilised by $s$, and $L$ would not be the
stabiliser of $LS^2$, since $s\not\in L$. Therefore, $\mathbf{A}_0\neq LS$.
For $\mathbf{B}_0$ we take any subspace of $LS$ such that 
$LS = \mathbf{A}_0 \oplus \mathbf{B}_0$, and we set $\mathbf{C}_0 = \{0\}$, so
that we have $LS=\mathbf{A}_0 \oplus \mathbf{B}_0 \oplus \mathbf{C}_0$, with
$a_0\geq 1$ and $b_0\geq 1$.
Note that conditions~\ref{item:cond1} and \ref{item:cond2} of the lemma are fulfilled by
construction. Condition~\ref{item:cond4} is trivially satisfied since $A_0^+ = S$.

There remains the issue of condition~\ref{item:cond3}: it is either satisfied, and we are done: or it is not, in which case we will look for another triple $(\mathbf{A}_1,\mathbf{B}_1,\mathbf{C}_1)$
(together with an element $s$)
satisfying all conditions of the Lemma except possibly the third one. More generally, given 
a triple $(\mathbf{A}_i,\mathbf{B}_i,\mathbf{C}_i)$ satisfying all the requirements except condition~\ref{item:cond3},
we now construct another triple $(\mathbf{A}_{i+1},\mathbf{B}_{i+1},\mathbf{C}_{i+1})$ satisfying all requirements except possibly condition~\ref{item:cond3}, and such that $\mathbf{A}_{i+1}\subsetneq \mathbf{A}_i$.

First observe that we have 
\[a_i = \dim_L \mathbf{A}_i\geq 2,\]
otherwise we would have $\mathbf{A}_i = L$ and
$A_i = S_0$, satisfying condition~\ref{item:cond3}.
Next, we change the current value of $s$ to a value $s \in \{s_2, \ldots , s_{\kappa}\}$ such
that $\mathbf{A}_i s \not\subset LS$. Such an $s$ exists because
otherwise $LS \mathbf{A}_i \subset LS$ and $LS$ would be stabilised by
$\mathbf{A}_i$, implying that the stabiliser of $LS^2$ would not be $L$ but a
field containing $\mathbf{A}_i$. We now set
$\mathbf{A}_{i+1} = \{\mathbf{a} \in \mathbf{A}_i, \mathbf{a}s \in LS\}
\subsetneq \mathbf{A}_i$ and take $\mathbf{B}_{i+1}$ to be any
complementary subspace of $\mathbf{A}_{i+1}$ in $\mathbf{A}_i$. Finally we set
$\mathbf{C}_{i+1} = \mathbf{B}_i \oplus \mathbf{C}_i$, so that 
$LS=\mathbf{A}_{i+1}\oplus\mathbf{B}_{i+1}\oplus\mathbf{C}_{i+1}$.

Note that $L\subset\mathbf{A}_{i+1}\subsetneq\mathbf{A}_i$, and that $b_{i+1}\geq 1$.
By definition of $A_{i+1}^+$, we have $A_{i+1}^+ = A_i$: since $A_i$ does not
satisfy condition~\ref{item:cond3}, we have
\[
\dim_K A_{i+1}^+ > \dim_K S_0 + (a_i - 1)\ell.
\]
From Proposition~\ref{prop:ScapL}, $\dim_K S_0 \geq \ell + 1 - \tau$ and, from \eqref{eq:taubis},
$\dim_K S = \kappa \ell + 1 - \tau$. Thus,
\[
\dim_K A_{i+1}^+ > \dim_K S - (\kappa - a_i)\ell.
\]
But $\kappa - a_i = b_i + c_i = c_{i+1}$: thus, the new triple
$(\mathbf{A}_{i+1}, \mathbf{B}_{i+1}, \mathbf{C}_{i+1})$ satisfies
condition~\ref{item:cond4}. Conditions~\ref{item:cond1} and \ref{item:cond2} are naturally satisfied by construction,
therefore we again have a triple satisfying all conditions of the Lemma except possibly for condition~\ref{item:cond3}.

As already observed, condition~\ref{item:cond3} is
trivially satisfied when $\mathbf{A}_i = L$. We have therefore constructed a sequence of triples
$(\mathbf{A}_i,\mathbf{B}_i,\mathbf{C}_i)$ that, since the $\mathbf{A}_i$ are strictly decreasing,
must eventually lead to $L$-spaces $\mathbf{A}, \mathbf{B}, \mathbf{C}$ and an
element $s$ satisfying all conditions of the lemma.
\end{proof}

From now on we let $\mathbf{A}, \mathbf{B}, \mathbf{C}$ and $s \in S$ be fixed
as guaranteed by Lemma~\ref{lem:ABC}. 

\Subsection{Studying $S_0S$ and $S_0s\cap S$}
Let us write $S_0 = S_0' \oplus S_0''$
where $S_0'$ is the subspace of $S_0$ consisting of the elements $y$ such
that $ys \in (S + xS)$, i.e. $S_0' = S_0 \cap s^{-1}(S + xS)$, and
$S_0''$ is any complementary subspace of $S_0'$ in~$S_0$. 
Since $1,x\in S_0$ we have $S_0S\supset (S+xS)$, and we also have $S_0S\supset S_0''s$. By definition of~$S_0''$ we have $S_0''s\cap(S+xS)=\{0\}$, therefore
\begin{equation}
\label{eq:S0S}
\dim_K S_0S \geq \dim_K(S + xS) + \dim_K S_0''.
\end{equation}
Note that $1, x \in S_0'$. Let $T_0$ be the subspace of elements of $S_0'$ whose valuation at $Q_\infty$ is larger than $-N=v_\infty(x)$. We clearly have $S_0'=T_0\oplus Kx$, so that $\dim_K T_0=\dim_K S_0'-1$.

\begin{lemma}
\label{lem:v(s)}
We have $T_0s\subset S.$
\end{lemma}

\begin{proof}
We have $T_0\subset S_0'$ and by definition of $S_0'$ we have have therefore $T_0s\subset S+xS$.
Suppose there exists $\lambda\in T_0$ such that $\lambda s\not\in S$. 
Then, since $\lambda s\in S+xS$, there must be some $y\in S$ such that $v_{Q_\infty}(\lambda s)=v_{Q_\infty}(yx)$, i.e.
\begin{equation}
\label{eq:xlambda}
v_{Q_\infty}(\lambda)+ v_{Q_\infty}(s)=v_{Q_\infty}(y)+v_{Q_\infty}(x).
\end{equation}
Since $x,\lambda\in L$, their valuations at $Q_\infty$ are in $d\Z$, so we have
$v_{Q_\infty}(y) = v_{Q_\infty}(s) \bmod d$. Now recall from Lemma~\ref{lem:ABC}
that $s$ is chosen to be one of the elements $s_1,\ldots ,s_\kappa$ and therefore 
(from the end of Section~\ref{sec:groundwork})
$v_{Q_\infty}(s)$ is maximal within its class modulo $d$. In other words we have $v_{Q_\infty}(y) \leq v_{Q_\infty}(s)$.
But the definition of $T_0$ implies $v_{Q_\infty}(x) < v_{Q_\infty}(\lambda)$, so we have a contradiction with \eqref{eq:xlambda}.
\end{proof}

Since $A^+\subset S$, Lemma~\ref{lem:v(s)} implies $T_0sA^+\subset S^2$, and we have
\begin{equation}\label{eq:sub_S0}
    S^2 \supset S_0S + T_0sA^+. 
\end{equation}
We will study the sum $S_0S + T_0sA^+$ and its dimension, aiming to derive a contradiction
with the Freiman hypothesis \eqref{eq:freiman}.  More precisely, since $A^+\subset\mathbf{A}+\mathbf{B}$, by condition \ref{item:cond1} of Lemma~\ref{lem:ABC} we have $A^+s\subset LS+\mathbf{B}s$: and since $T_0\subset L$ and $LS+\mathbf{B}s$ is $L$-linear,
we have
\begin{equation}
    \label{eq:T0sA+}
T_0sA^+ \subset LS + \mathbf{B}s,
\end{equation}
and finally, since $S_0\subset L$, we also have $S_0S + T_0sA^+ \subset LS + \mathbf{B}s$.
Therefore,
\[
  \dim_L L(S_0S + T_0sA^+) \leq \kappa + b
\]
(recall that $b = \dim_L \mathbf{B}$). 
Applying Lemma~\ref{lem:W} to \eqref{eq:sub_S0}, we therefore get:
\begin{equation}
\label{eq:final}
\dim_K S^2 \geq \dim_K(S_0S + T_0sA^+) + (\kappa - 1 - b)\ell.
\end{equation}
We will find a contradiction with the Freiman hypothesis as soon as we
show that
\[
\dim_K (S_0S + T_0sA^+) \geq 3\dim_K S - 3 - (\kappa - 1 - b)\ell,
\]
which is now our objective. The dimension of $S_0S + T_0sA^+$ will be bounded from below by the sum of two terms: 
the dimension of $S_0S$ given by \eqref{eq:S0S} and the contribution of elements of $T_0sA^+$ that
``fall outside of $LS$'', in other words, those that have a nonzero
component in $\mathbf{B}s$.

\Subsection{The lower bound on $S_0S+T_0sA^+$ that contradicts the Freiman hypothesis}

Consider the projection map
\[
\Phi~:~\map{LS\oplus \mathbf{B}s}{\mathbf{B}s}{\bm{\sigma}+\mathbf{b}s}{\mathbf{b}s}
\]
where $\bm{\sigma}$ and $\mathbf{b}$ are arbitrary elements of $LS$
and $\mathbf{B}$ respectively. The following lemma is a simple remark, but is a key to the argument to follow.
\begin{lemma}
\label{lem:Phi}
For any finite-dimensional $K$-subspace $V$ of $LS+\mathbf{B}s$ and any finite-dimensional $K$-subspace $U$ of
$L$, we have $\Phi(UV)=U\Phi(V)$.
\end{lemma}
\begin{proof}
This is a simple application of the $L$-linearity of the map $\Phi$.
\end{proof}

All dimensions below are $K$-dimensions. Let us write
\begin{align}
\nonumber \dim(S_0S+T_0sA^+) &\geq \dim S_0S+ \dim T_0sA^+ -\dim(S_0S\cap T_0sA^+)\\
\label{eq:minor_dim}                   &\geq \dim S_0S+ \dim T_0sA^+ -\dim (LS\cap T_0sA^+)
\end{align}
since $S_0S\subset LS$. From \eqref{eq:T0sA+} and the rank-nullity theorem we have 
\[
\dim T_0sA^+=\dim\Phi(T_0sA^+) + \dim(\ker\Phi\cap T_0sA^+).
\]
Therefore, since $LS=\ker\Phi$, inequality \eqref{eq:minor_dim} yields
\[
\dim(S_0S+T_0sA^+) \geq \dim S_0S + \dim\Phi(T_0sA^+).
\]
Applying Lemma~\ref{lem:Phi} we therefore have, since $T_0\subset L$,
\begin{equation}
\label{eq:S0S+}
\dim(S_0S+T_0sA^+) \geq \dim S_0S + \dim T_0\Phi(sA^+).
\end{equation}
Since $A^+s\subset LS+\mathbf{B}s$, applying again the rank-nullity theorem,
\begin{align*}
\dim\Phi(A^+s)&=\dim A^+s -\dim (\ker\Phi\cap A^+s)\\
              &=\dim A^+ - \dim (LS\cap A^+s).
\end{align*}
But from the definition of $A^+=S\cap(\mathbf{A}+\mathbf{B})$ and conditions \ref{item:cond1} and \ref{item:cond2} of Lemma~\ref{lem:ABC},
we have $LS\cap A^+s=As$. Therefore,
\[
\dim\Phi(A^+s) = \dim A^+ - \dim A.
\]
Lemma~\ref{lem:CD} now gives us the ``Cauchy-Davenport inequality'':
\begin{align*}
\dim T_0\Phi(sA^+) &\geq \dim T_0 + \dim \Phi(A^+s) - 1\\
                   &\geq \dim T_0 + \dim A^+-\dim A -1
\end{align*}
by the above computation of $\dim\Phi(A^+s)$.
Substituting into inequality \eqref{eq:S0S+} we therefore have:
\[
\dim(S_0S + T_0sA^+) \geq \dim S_0S + \dim T_0 + \dim A^+ - \dim A - 1.
\]
Applying \eqref{eq:S0S} we now obtain:
\begin{align*}
\dim(S_0S + T_0sA^+) &\geq \dim(S + xS) + \dim S_0'' + \dim T_0 + \dim A^+ - \dim A - 1 \\
                     &= \dim(S + xS) + \dim S_0'' + \dim S_0' + \dim A^+ - \dim A - 2 \\
                     &= \dim(S + xS) + \dim S_0 + \dim A^+ - \dim A - 2
\end{align*}
where we have used $S_0=S_0'\oplus S_0''$ and $\dim T_0=\dim S_0'-1$.
Applying conditions~\ref{item:cond3} and \ref{item:cond4} from Lemma \ref{lem:ABC}, we obtain:
\begin{align*}
\dim(S_0S + T_0sA^+) &\geq \dim(S + xS) - 2 + \dim S - c\ell - (a - 1)\ell \\
                     &\geq 2\dim S - 1 - 2 + \dim S - (\kappa - 1 - b)\ell \\
                     &= 3\dim S - 3 - (\kappa - 1 - b)\ell
\end{align*}
where we have used $\dim_LLS=\kappa=a+b+c$.
Finally, substituting into \eqref{eq:final} gives the desired contradiction
with the Freiman hypothesis~\eqref{eq:freiman}. This concludes the proof of Theorem~\ref{thm:main} for algebraically closed fields $K$.

 \section{When $K$ is not algebraically closed}
\label{sec:perfect_field}
Once Theorem~\ref{thm:main} is proved in the special case of an algebraically closed ground field~$K$, there remains to extend it to the more general case a perfect ground field $K$. Suppose that $K$ is perfect and algebraically closed
in $F$. The proof to follow is in the same spirit as that of \cite[\S~9]{bcz18}.
Denote by $\barK$ the algebraic closure of $K$. Since~$K$ is algebraically closed in $F$,
the compositum $F' \eqdef \barK F$ is a function field over $\barK$. Moreover, from \cite[Prop.~3.6.1(b)]{S09},
$\dim_{\barK} \barK T = \dim_K T$ for any $K$-subspace $T$, so that we have
\begin{equation}\label{eq:dim_barKS}
    \dim_{\barK} \barK S = \dim_K S
\end{equation}
and also 
\[
\dim_{\barK} \barK S^2 = 2 \dim_{\barK} \barK S - 1 + \gamma.
\]
Theorem~\ref{thm:main} for algebraically closed fields therefore applies, so that
$F'$ has genus $g \leq \gamma$ and the least divisor $D'$ of $F'$ such that 
$\barK S \subset \L_{F'} (D')$ satisfies
\begin{equation}\label{eq:result_alg_closed}
\dim_{\barK} \L_{F'} (D') - \dim_{\barK} S \leq \gamma - g.
\end{equation}

From \cite[Thm.~3.6.3(b)]{S09}, $F'$ has the same genus as $F$ which yields the upper
bound on the genus. Let $D$ be the least divisor of $F$ such that $S \subset \L_F (D)$, there remains to evaluate the codimension of $S$ in $\L_F (D)$. 
Following the lines of \cite[Thm.~9.2]{bcz18},
one proves easily that $D'$ is nothing but the \emph{conorm} (also called the \emph{pullback})
of $D$ in $F'$ (see \cite[Def.~3.1.8]{S09}) and, from \cite[Thm.~3.6.3(d)]{S09},
\begin{equation}\label{eq:dimLD'}
    \dim_{\barK} \L_{F'}(D') = \dim_{K} \L_{F} (D).
\end{equation}
Therefore, \eqref{eq:dim_barKS} and \eqref{eq:dimLD'} together with \eqref{eq:result_alg_closed}
permits to conclude.

\begin{remark}
Note that the assumption ``$K$ is perfect'' can be relaxed to ``any finite extension of $K$ is primitive''. This is indeed the only requirement for the above argument to hold. Therefore, Theorem~\ref{thm:main} remains true, for example, when $K$ is a one variable function field over
a finite field.
\end{remark}

\section*{Acknowledgements}
The authors express their deep gratitude to Christine Bachoc and Hugues Randriambololona for inspiring discussions. This work was supported by
the French \emph{Agence Nationale de la Recherche} project ANR-21-CE39-0009 \emph{Barracuda} and by Horizon-Europe MSCA-DN project \emph{Encode}.
  
\bibliographystyle{alpha}

\end{document}